\documentclass[12pt]{amsart}

\usepackage{amssymb,latexsym}
\usepackage{amsmath,amsfonts,amsthm,amscd,amsxtra,enumerate}
\usepackage{mathtools}
\usepackage[all]{xy}
\usepackage{amsmath,latexsym,amssymb, enumerate, amsthm, epsfig} 

\usepackage[margin=1in]{geometry}
\usepackage{epsfig}

\usepackage{ytableau}

\usepackage[dvipdfmx]{hyperref}

\usepackage{mathptmx}

\usepackage[pagewise]{lineno}

\theoremstyle{plain}
\newtheorem{theorem}{Theorem}[section]
\newtheorem{proposition}[theorem]{Proposition}

\newtheorem{corollary}[theorem]{Corollary}
\theoremstyle{definition}

\newtheorem{setup}[theorem]{Setup}

\newtheorem{definition}[theorem]{Definition}
\newtheorem{example}[theorem]{Example}

\theoremstyle{remark}

\newcommand{\sk}{{\ensuremath{\sf k }}}
\newcommand{\m}{\ensuremath{\mathfrak m}}

\def\m{\mathfrak m}

\DeclareMathOperator{\rank}{rank}

\DeclareMathOperator{\Hom}{Hom}

\begin{document}

\title[Almost Gorenstein determinantal rings]{Almost Gorenstein determinantal rings \\ of symmetric matrices}

\author[Ela Celikbas]{Ela Celikbas}
\address{Department of Mathematics, West Virginia University, Morgantown, WV 26506}
\email{ela.celikbas@math.wvu.edu}

\author[Naoki Endo]{Naoki Endo}
\address{Department of Mathematics, Faculty of Science, Tokyo University of Science, 1-3 Kagurazaka, Shinjuku, Tokyo 162-8601, Japan}
\email{nendo@rs.tus.ac.jp}
\urladdr{https://www.rs.tus.ac.jp/nendo/}

\author[Jai Laxmi]{Jai Laxmi}
\address{School of Mathematics, Tata Institute of Fundamental Research,
Mumbai-400005, India.}
\email{laxmiuohyd@gmail.com}
\email{jailaxmi@math.tifr.res.in}

\author[Jerzy Weyman]{Jerzy Weyman}
\address{Instytut Matematyki, Jagiellonian University, Krak\'ow 30--348}
\email{jerzy.weyman@uj.edu.pl}

\thanks{N. Endo was partially supported by JSPS Grant-in-Aid for Young Scientists 20K14299. J. Laxmi was supported by Fulbright-Nehru fellowship. J. Weyman was partially supported by NSF grant DMS 1802067, Sidney Professorial Fund, and Polish National Agency for Academic Exchange.}

\begin{abstract}
We provide a characterization of the almost Gorenstein property of determinantal rings of a symmetric matrix of indeterminates over an infinite field. We give an explicit formula for ranks of the last two modules in the resolution of determinantal rings using Schur functors. 
\end{abstract}

\maketitle

\setlength{\baselineskip}{15.5pt}

\section{Introduction}\label{Intro}

{\it An almost Gorenstein ring} is, one of the good candidates for generalization of Gorenstein rings, defined by an existence of embedding of the rings into their canonical modules whose cokernel is {\it an Ulrich module}, i.e, the multiplicity is equal to the number of generators. The motivation of this generalization comes from the strong desire to stratify Cohen-Macaulay rings, finding new and interesting classes which naturally extend the Gorenstein rings. The theory of almost Gorenstein rings was introduced by Barucci and Fr{\"o}berg \cite{BF97} in the case where the local rings are analytically unramified and of dimension one. Their work inspired Goto, Matsuoka, and Phuong \cite{GMP} to give a modified definition of the one-dimensional almost Gorenstein local rings in 2013. Two years later Goto, Takahashi and the second author of this paper \cite{GTT} defined the almost Gorenstein graded/local rings of arbitrary dimension. In 2017, the second author studied the question of when the determinantal rings of generic matirices are almost Gorenstein rings. 
The goal of this paper is to provide a characterization of the almost Gorenstein property of determinantal rings of a symmetric matrix. 


Let $\sk$ be an infinite field and let $n>0$ be an integer. Let $X=[X_{ij}]$ be a generic $n\times n$ symmetric matrix of indeterminates over $\sk$, i.e., $X_{ij}=X_{ji}$ for all $1 \leq i, j \leq n$. We denote by $S=\sk[X]$ the polynomial ring over $\sk$ generated by $n(n+1)/2$ variables $\{X_{ij}\}_{1 \leq i, j \leq n}$.  We  consider $S$ as a $\Bbb Z$-graded ring under the grading $S_0 = \sk$ and $X_{ij} \in S_1$ for all $1 \leq i, j \leq n$. 
Let $I_{t+1}(X)$ be the ideal of $S$ generated by $(t+1)\times(t+1)$ minors of the matrix $X$ where $0< t < n$ is an integer. We set 
$
R=S/I_{t+1}(X)
$
and call it {\it the determinantal ring of $X$}. By Kutz \cite[Theorem 1]{Kutz}, the ring $R$ is known to be a Cohen-Macaulay integral domain with $\dim R=nt-\frac{1}{2}t(t-1)$. Goto shows in \cite{Goto79} that $R$ is Gorenstein if and only if $n-t$ is odd.


With this notation, we state the main result of this paper below, where $\mathfrak m=R_{+}$ denotes the graded maximal ideal of $R$.

\begin{theorem}\label{mainthm}
The following conditions are equivalent. 
\begin{enumerate}
\item[$(1)$] $R$ is an almost Gorenstein graded ring.
\item[$(2)$] $R_{\mathfrak m}$ is an almost Gorenstein local ring.
\item[$(3)$] Either $n-t$ is odd, or $n=3$, $t=1$.
\end{enumerate}
\end{theorem}

 Theorem~\ref{mainthm} gives the following invariant-theoretic application. If the field $\sk$ has characteristic 0, then the determinantal ring appears as a ring of invariants. Suppose that $V$ is an $n\times t$ matrix of indeterminates over $\sk$ and $A=\sk[V]$ is the polynomial ring over $\sk$. Let $G=O(t,\sk)$ be the orthogonal group. Assume that the group $G$ acts on the ring $A$ as $\sk$-automorphisms by taking $V$ onto $VH^{-1}$ for every $H\in G$. Then the ring $A^G$ of invariants is generated by the entries of the $n\times n$ symmetric matrix $Y=VV^T$ and the ideal of relations on $Y$ is generated by the $(t+1)\times (t+1)$ minors of $Y$; see \cite{CP76}. Hence we get the following. 
 
\begin{corollary}
Let $A$ and $G$ be given as above. Then $A^G=\sk[Y]$ is an almost Gorenstein graded ring if and only if either $n-t$ is odd, or $n=3$, $t=1$.
\end{corollary} 

This paper is organized as follows. In Section $2$ we give some basic properties on almost Gorenstein rings. We prove Theorem \ref{mainthm} in Section 3. We also explain a rank computation of the last two modules of the resolution of determinantal rings using the techniques in the representation theory of finite groups.

  
\section{Preliminaries}\label{Prelim}
In this section we list the definition and the basic properties of almost Gorenstein rings, which we will use throughout this paper. Let $(R, \m)$ be a Cohen-Macaulay local ring with $d=\dim R$, possessing a canonical module $K_R$.

\begin{definition}({\cite[Definition 1.1]{GTT}})\label{3.1}
We say that $R$ is {\it an almost Gorenstein local ring}, if there exists an exact sequence
$$
0 \to R \to K_R \to C \to 0
$$
of $R$-modules such that $\mu_R(C) = e^0_{\m}(C)$, where $\mu_R(C)$ denotes the number of elements in a minimal system of generators for $C$ and
$$
e^0_{\m}(C) = \lim_{n \to \infty}\frac{\ell_R(C/\m^{n+1}C)}{n^{d-1}}\cdot (d-1)!
$$
is the multiplicity of $C$ with respect to $\m$.
\end{definition}

Every Gorenstein ring is an almost Gorenstein ring, and the converse holds if the ring $R$ is Artinian (\cite[Lemma 3.1 (3)]{GTT}). Definition \ref{3.1} insists that an almost Gorenstein ring $R$ might not be Gorenstein, but the ring $R$ can be embedded into its canonical module $K_R$ and the difference $K_R/R$  has good properties.
For an arbitrary exact sequence
$$
0 \to R \to K_R \to C \to 0
$$
of $R$-modules with $C \ne (0)$, the $R$-module $C$ is Cohen-Macaulay  and of dimension $d-1$ (\cite[Lemma 3.1 (2)]{GTT}).
Suppose that $R$ has an infinite residue class field $R/ \m$. Consider the local ring $R_1=R/[(0):_RC]$ with maximal ideal $\m_1$. 
We can choose elements $f_1, f_2, \ldots, f_{d-1} \in \m$ satisfying the ideal $(f_1, f_2, \ldots, f_{d-1})R_1$ forms a minimal reduction of $\m_1$. Then
$$
e_{\m}^0(C) = e_{\m_1}^0(C) = \ell_R(C/(f_1, f_2, \ldots, f_{d-1})C) \ge \ell_R(C/\m C) = \mu_R(C). 
$$
Therefore, $e_{\m}^0(C) \ge \mu_R(C)$ and we say that $C$ is {\it an Ulrich $R$-module} if $e_{\m}^0(C) = \mu_R(C)$. Thus, $C$ is an Ulrich $R$-module if and only if $\m C = (f_1, f_2, \ldots, f_{d-1})C$. 




In the rest of this section, let $R=\bigoplus_{n\ge 0}R_n$ be a Cohen-Macaulay graded ring. Assume $\sk=R_0$ is a local ring and there exists the graded canonical module $K_R$. Set $a=a(R)$ the $a$-invariant of $R$, i.e., 
$
a=\max\{n\in \mathbb{Z} \mid [H^d_{\mathfrak{M}}(R)]_n\neq 0\}
$
where $\{[H^d_{\mathfrak{M}}(R)]_n\}_{n\in \mathbb{Z}}$ denotes the homogeneous components of $d$-th graded  local cohomology module of $H^d_{\mathfrak{M}}(R)$ of $R$ with respect to the unique graded maximal ideal $\mathfrak{M}$. 

\begin{definition}(\cite[Definition 1.5]{GTT})
We say that $R$ is {\it an almost Gorenstein graded ring}, if there exists an exact sequence
$$
0 \to R \to K_R(-a) \to C \to 0
$$
of graded $R$-modules such that $\mu_R(C) = e_{\mathfrak{M}}^0(C)$.
\end{definition}

Note that $K_R(-a)$ stands for the graded $R$-module whose underlying $R$-module is the same as that of $K_R$ and whose grading is given by $[K_R(-a)]_n=[K_R]_{n-a}$ for all $n\in \mathbb{Z}$. 
Similarly for local case, every Gorenstein graded ring is an almost Gorenstein graded ring. 
As $C_\mathfrak{M}$ is Ulrich as an $R_\mathfrak{M}$-module  and the canonical module $K_R$ is compatible with localization, the local ring $R_{\mathfrak{M}}$ is almost Gorenstein, provided $R$ is almost Gorenstein as a  graded ring. Although the converse doesn't hold in general (see e.g., \cite[Theorems 2.7, 2.8]{GMTY2}, \cite[Example 8.8]{GTT}), the theory is still attractive and worth studying the theory as, for example, the following theorem shows.


\begin{theorem}{$($\cite[Theorem 1.6]{GTT}$)$}\label{thm2}
Let $R=\sk[R_1]$ be a Cohen-Macaulay homogeneous ring over an infinite residue field $\sk$. Suppose that $R$ is not a Gorenstein ring. Then the following conditions are equivalent.
\begin{enumerate}
\item[$(1)$] $R$ is almost Gorenstein and level. 
\item[$(2)$] The field of fractions $Q(R)$ of $R$ is Gorenstein and $a(R)=1-\dim R$.
\end{enumerate}
\end{theorem}

We can directly apply Theorem \ref{thm2} to the determinantal rings of arbitrary matrices.  For instance, we have the following consequences. 

\begin{corollary}
Let $\sk$ be an infinite field and $n \ge 3$ be an odd integer. Let $X=[X_{ij}]$ be a generic $n\times n$ skew-symmetric matrix over $\sk$, i.e.,  $X_{ij}=-X_{ji}$ for all $1 \leq i, j \leq n$. Let $I$ be the ideal generated by submaximal Pfaffians of $X$. Then $R=k[X]/I^2$ is an almost Gorenstein graded ring if and only if $n=3$. 
\end{corollary}

\begin{proof}
Note that $R=k[X]/I^2$ is a Cohen-Macaulay ring with $\dim R={n \choose 2} -3$ possessing the projective dimension 3 (\cite[Theorem 2.5]{Boffi-Sanchez}). By \cite{Perlman}, $I^2$ has linear minimal free resolution. We have the Betti numbers as $\beta_1(I^2)={n+1 \choose 2}$, $\beta_2(I^2)=n^2-1$, and $\beta_3(I^2)={n \choose n-2}$. By setting $n=2k+1$ with $k \ge 1$, the sequence
$$
0 \to S^{\oplus {n\choose 2}}(-2k-2) \to S^{\oplus (n^2-1)}(-2k-1) \to S^{\oplus {n+1\choose 2}}(-2k)\to S \to R \to 0
$$
forms a minimal $S$-free resolution of $R$, where $S=k[X]$. 
Hence the $a$-invariant is $2k+2=n+1$, so that the assertion follows from  Theorem~\ref{thm2}. 
\end{proof}

Let $1 \le t \le n$ be integers and $X_1, \ldots, X_{n+t-1}$ be indeterminates over an infinite field $\sk$. {\it A Hankel determinantal ring} is a ring of the form 
$$
R=\sk[X_1, \ldots, X_{n+t-1}]/ I_t(H)
$$
where $I_t(H)$ denotes the ideal of the polynomial ring $\sk[X_1, \ldots, X_{n+t-1}]$ generated by $t \times t$ minors of {\it the Hankel matrix $H$}  
$$
H=\begin{bmatrix*}
X_1 & X_2 & X_3 & \cdots & X_s \\
X_2 & X_3 & X_4 & \cdots & X_{s+1}\\
\vdots & \vdots & \vdots & \ddots & \vdots \\
X_r & X_{r+1} & \cdots & \cdots & X_{s+r-1}
 \end{bmatrix*}.
 $$
We now consider the generic determinantal ring $A=\sk[Y]/I_t(Y)$, where $Y$ is a $t\times n$ matrix of indeterminates over $\sk$, and $I_t(Y)$ denotes the ideal of $\sk[Y]$ generated by $t \times t$-minors of $Y$. Since the coset of elements $Y_{i, j+1}-Y_{i+1, j}$ are part of a system of parameters for $A$, we have a ring isomorphism between $A$ modulo these elements and $R=\sk[X_1, \ldots, X_{n+t-1}]/ I_t(H)$. Therefore, $R$ is a Cohen-Macaulay ring with $\dim R = 2t-2$, and hence $R$ is Gorenstein if and only if $n=t$. 

\begin{corollary}
$R$ is an almost Gorenstein graded ring if and only if $n=t$ or $n\neq t$, $t=2$.
\end{corollary}

\begin{proof}
The coset of elements $X_1, \ldots, X_{t-1}, X_{n+1}, \ldots, X_{n+t-1}$ are a homogeneous system of parameters for $R$, and the socle modulo this system of parameters is spanned by the degree $t-1$ monomials in $X_t, \ldots, X_n$. Hence $a(R) =1-t$. By Theorem~\ref{thm2}, we get the required assertion.  
\end{proof}

\section{Proof of Theorem \ref{mainthm}} 

This section aims at proving Theorem \ref{mainthm}. First of all, we fix our notation and assumptions on which all the results in this section are based.

\begin{setup}\label{Setup1}
Let $\sk$ be an infinite field, and $0<t<n$ be integers. Let $X=[X_{ij}]$ be a symmetric matrix of indeterminates over $\sk$, and $S= \sk[X](=\sk[X_{ij}]_{1\leq i\leq j\leq n})$ be the polynomial ring over $\sk$ in $n(n+1)/2$ variables. We set $R=S/I_{t+1}(X)$, where $I_{t+1}(X)$ denotes the ideal of $S$ generated by $(t+1)\times(t+1)$ minors of $X$. Let $\mathfrak{M}=(x_{ij} \mid 1\leq i\leq j \leq n)$ where $x_{ij}$ stands for the image of $X_{ij}$ in $R$. Set $d=\dim R=nt-\frac{1}{2}t(t-1)$. 
\end{setup}

Until Proposition \ref{3.3}, we assume ${\rm char}(\sk)=0$ and $R=S/{I}_{t+1}(X)$ is not a Gorenstein ring. Let 
$$
0 \to F_{\ell} \to F_{{\ell}-1} \to \cdots \to F_0 \to R \to 0
$$
be a minimal $S$-free resolution of $R$. 
The key of the proof of Theorem \ref{mainthm} is the equalities
$$
\rank F_{{\ell}-1}=n\dbinom{n}{t+1}-\dbinom{n}{t+2} \ \ \ \text{and} \ \ \  \rank F_{\ell}=\dbinom{n}{t}.
$$
To show this, we need some techniques from representation theory of finite groups. More precisely, let $X^s$ be the space of $n\times n$ symmetric matrices over $\sk$. The coordinate ring $A=\sk[X^s]$ of $X^s$ is isomorphic to the polynomial ring $\sk[\phi_{i,j}]_{1\leq i\leq j\leq n}$, where $\phi_{i,j}$ denotes the $(i, j)$-th coordinate function on $X^s$ with $\phi_{i,j} = \phi_{j, i}$. 

We identify $X^s$ with the space $S_2E^*$, {\it the $2$nd symmetric power of $E^*$}, where $E$ is the vector space over  $\sk$  with $\dim_{\sk}E=n$ and $(-)^* = \Hom_\sk(-, \sk)$ is the dual functor. Let us also identify the coordinate ring $A=\sk[X^s]$ of $X^s$ with the symmetric algebra $\text{Sym}(S_2E^*)$. By setting $\{e_i\}_{1 \le i \le n}$ the $\sk$-basis of $E$, the coordinate function $\phi_{i,j}$ can be identified with the element $e_ie_j$ in $\text{Sym}(S_2E^*)$. 

For each integer $t$ with $0 < t < n$, let us consider the subvariety
$$
Y^s_t=\{\phi \in X^s\mid \text{rank}\; \phi \leq t \}
$$
of $X^s$, called {\it the symmetric determinantal variety}. 
Note that the variety $Y^s_t$ can be identified with the set of symmetric matrices $\Phi$ whose minors of size $t+1$ vanish. 
Hence, in order to compute the ranks of the tail and its before one of the resolution of determinantal rings $R=S/I_{t+1}(X)$, it comes down to think about the resolution of $A/J_{t+1}$, where $J_{t+1}$ denotes the ideal of $A$ generated by $(t+1) \times (t+1)$ minors of $\Phi$. 


Let $\lambda=(\lambda_1, \lambda_2, \ldots, \lambda_m)$ be a partition of a positive integer $n$, i.e., a weakly decreasing sequence $\lambda_1 \ge \lambda_2 \ge \cdots \ge \lambda_m$ of non-negative integers such that $|\lambda|=\sum_{i=1}^m \lambda_i = n$, where $|\lambda|$ is called {\it the weight of $\lambda$}. 
We identify the partitions $(\lambda_1, \lambda_2, \ldots, \lambda_m)$ and $(\lambda_1, \lambda_2, \ldots, \lambda_m, 0)$. 
For each partition, we associate its  Young diagram, so that we now identify the partition with its Young diagram. 
The length of the diagonal of $\lambda$ is called  {\it the rank of $\lambda$}, denoted by $\rank \lambda$. In other words, $\rank \lambda = r$ is the biggest $r\times r$ square fitting inside $\lambda$. 
For each box $\mathbb{X}$ in a Young diagram $\lambda$, the set of boxes to the right (resp. below) of $\mathbb{X}$ (including $\mathbb{X}$) is called {\it an arm of $\mathbb{X}$} (resp. {\it a leg of $\mathbb{X}$}). 
{\it A hook of $\mathbb{X}$} consists of the arm and leg of $\mathbb{X}$, and the number of boxes in the hook of $\mathbb{X}$ is called {\it the hook length of $\mathbb{X}$.} 

Let $\lambda$ be a partition of rank $r$. For each $1 \le i \le r$, we denote by $a_i$ (resp. $b_i$) {\it the arm length} (resp. {\it the leg length}), i.e., the number of boxes in the arm (resp. the leg), of the $i$-th box on the diagonal of $\lambda$. Since the partition is uniquely determined by its rank $r$ and the numbers $a_i$, $b_i$ for all $1 \le i \le r$, we can write
$$
\lambda = (a_1, a_2, \ldots, a_r \mid b_1, b_2, \ldots, b_r)
$$
which is called {\it the hook notation} for $\lambda$. Notice that, in the hook notation, we have $a_1>a_2> \cdots> a_r>0$, $b_1>b_2>\cdots>b_r>0$, and $|\lambda| = \sum_{i=1}^r(a_i + b_i)-r$. 
Let $\lambda =(\lambda_1, \lambda_2, \ldots, \lambda_m)$ be a partition. 
{\it The conjugate partition $\lambda'$} of $\lambda$ is defined by 
$$
\lambda'_i = \#\{t \mid 1 \le t \le m, \ \lambda_t \ge i\}.
$$
This means the Young diagram of $\lambda'$ is obtained from the Young diagram of $\lambda$ by reflecting at the line $y=-x$ in the coordinate plane. The reader may consult with \cite[Chapter 1]{W} for more details regarding partitions.

We denote by $Q_{t-1}(2m)$ the set of partitions $\lambda$ of rank $r$ that can be written as
$$
\lambda=(a_1, a_2, \ldots, a_{r} \mid b_1, b_2, \ldots, b_{r})
$$
in the hook notation, where $a_i=b_i+(t-1)$ for each $1 \le i \le r$ and $|\lambda|=2m$.  
This partition with even rank can be put in the form 
$$
\lambda=\lambda(\alpha, u)=(\alpha_1+2u+t-1, \ldots, \alpha_{2u}+2u+t-1, \alpha'_1, \ldots, \alpha'_v)
$$ 
where $u=\frac{1}{2} \rank \lambda$ and $\alpha$ is a partition with $\alpha_1'\leq  \rank \lambda$. Therefore, for example, the Young diagram of the partition $\lambda$ in $Q_{t-1}(2m)$ with even rank has the following form
$$
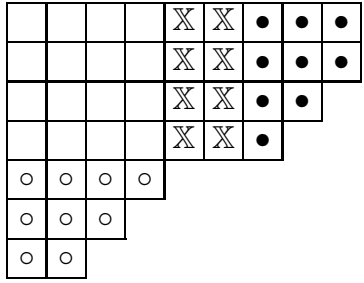
\ytableausetup{mathmode, boxsize=1.2em}
\begin{ytableau}
 {}  &  &  & & \mathbb{X}& \mathbb{X}&\bullet & \bullet&\bullet \\
   &  &    &  & \mathbb{X}& \mathbb{X}& \bullet& \bullet&\bullet \\
   &  & & & \mathbb{X}& \mathbb{X}& \bullet&\bullet \\
   &  & & & \mathbb{X} & \mathbb{X} & \bullet\\
 \circ   &  \circ& \circ& \circ  \\
 \circ   & \circ &  \circ \\
  \circ  & \circ    \\
\end{ytableau}
$$
where the boxes corresponding to $\alpha$ are filled by $\bullet$, the boxes corresponding the conjugate partition $\alpha'$ are filled by $\circ$, and the boxes providing additional $t-1$ elements for diagonal hook lengths are filled by $\mathbb{X}$.

For a partition $\lambda =(\lambda_1, \lambda_2, \ldots, \lambda_m)$, we consider {\it the Schur module} $L_{\lambda} E$, i.e., 
$$
L_{\lambda}E = \left(\bigwedge^{\lambda_1} E \otimes_{\sk} \bigwedge^{\lambda_2} E \otimes_{\sk} \cdots \otimes_{\sk}  \bigwedge^{\lambda_m} E\right)/R(\lambda, E)
$$
where the submodule $R(\lambda, E)$ is the sum of submodules of the form:
$$
\bigwedge^{\lambda_1} E \otimes_{\sk} \cdots \otimes_{\sk}  \bigwedge^{\lambda_{a-1}} E \otimes_{\sk} R_{a, a+1}(E) \otimes_{\sk}\bigwedge^{\lambda_{a+2}} \cdots \otimes_{\sk}  \bigwedge^{\lambda_{m}} E
$$
for $1 \le a \le m-1$. Here, $R_{a, a+1}(E)$ denotes the submodules of $\bigwedge^{\lambda_a} E \otimes_{\sk} \bigwedge^{\lambda_{a+1}} E $ spanned by the  images of the composite maps with $u+v < \lambda_{a+1}$ below:
$$
\xymatrix{
\bigwedge^{u} E \otimes_{\sk} \bigwedge^{\lambda_a-u + \lambda_{a+1} -v} E \ar[d]^{{\rm id} \otimes \Delta \otimes {\rm id}}  \otimes_{\sk} \bigwedge^{v} E \\ 
 \bigwedge^{u} E \otimes_{\sk} \bigwedge^{\lambda_a-u}E \otimes_{\sk} \bigwedge^{\lambda_{a+1} -v} E \ar[d]^{m_{12} \otimes m_{34}} \otimes_{\sk} \bigwedge^{v} E \\
\hspace{1.2em} \bigwedge^{\lambda_a} E \otimes_{\sk} \bigwedge^{\lambda_{a+1}} E
}
$$
where $\Delta$ stands for the diagonal map, and $m_{12}$ and $m_{34}$ are the multiplication maps.

\medskip 

With this notation above, we recall the following theorem that shows how to find the components of minimal free resolutions of $A/J_{t+1}$ in terms of $L_{\lambda} E$ corresponding to $\lambda$. 

\begin{theorem}$(${\cite[Theorem 6.3.1 (c)]{W}}$)$\label{rankthm} The i-th term $F_i$ of the minimal free resolution of $A/J_{t+1}$ as an $A$-module is given by the formula 
$$F_i=\bigoplus_{\lambda\in Q_{t-1}(2m),\; \rank \lambda \; \text{even},\; i=m-t\cdot\frac{1}{2}\cdot\rank \lambda} L_{\lambda}E \otimes_{\sk} A.
$$
\end{theorem} 
Note that the representations occurring in the resolution of $A/J_{t+1}$ are the Schur modules $L_{\lambda(\alpha, u)} E$ for all choices of $\alpha$ and $u$. 
The term  $L_{\lambda(\alpha, u)}E$ occurs in the $i$-th term in the resolution where $i=\frac{1}{2}(|\lambda| -t \cdot\rank \lambda)$.
For all $\lambda = \lambda(\alpha, u) \in Q_{t-1}(2m)$ with even rank $2u$, we have
\begin{center}
$|\lambda| = (2u)^2+2u(t-1) + |\alpha| + |\alpha'| = (2u)^2+2u(t-1) + 2|\alpha|$  
\end{center}
so that $m = \frac{1}{2}|\lambda| = 2u^2+u(t-1) + |\alpha|$.   
Hence the term occurs in $F_i$
with $i=2u^2 - u + |\alpha|$.
\medskip

The following is a key in the proof of Theorem \ref{mainthm}.

\begin{proposition} \label{3.3}
Let $\lambda(\alpha_{\ell}, u_{\ell})$ and $\lambda(\alpha_{\ell-1}, u_{\ell-1})$ be partitions associated with $F_{\ell}$ and $F_{\ell-1}$, respectively. Then the following assertions hold true. 
\begin{enumerate}
\item[$(1)$] $\rank \lambda(\alpha_{\ell}, u_{\ell}) = \rank \lambda(\alpha_{\ell-1}, u_{\ell-1})=n-t$. 
\smallskip
\item[$(2)$] $\rank F_{\ell}=\dbinom{n}{t}$ and $\rank F_{{\ell}-1}=n\dbinom{n}{t+1}-\dbinom{n}{t+2}$.
\end{enumerate}
\end{proposition}

\begin{proof} 
$(1)$~We consider a partition  $\lambda(\alpha_{\ell}, u_{\ell})$ associated with $F_{\ell}$ of rank $2u_{\ell}$ for some $0 < u_{\ell}\in \mathbb Z$. Suppose that $2u_{\ell}> n-t$. Since $n-t$ is even, we write $2u_{\ell} =n-t+2k$ with $k >0$. The projective dimension $\ell$ of $A/J_{t+1}$ as an $A$-module is given by 
$$
\ell=\frac{n(n+1)}{2}-nt+\frac{t(t-1)}{2}. 
$$
Hence the weight of $\alpha_{\ell}$ is  
\begin{eqnarray*}
|\alpha_{\ell}| &=& \ell - 2u_{\ell}^2 + u_{\ell} = \left(\frac{n(n+1)}{2}-nt+\frac{t(t-1)}{2}\right)  - \frac{1}{2} \left(4u_{\ell}^2 - 2u_{\ell}\right) \\
&=& (n-t + k)(1-2k) < 0
\end{eqnarray*}
which makes a contradiction. Thus we get $\rank \lambda(\alpha_{\ell}, u_{\ell})= 2u_{\ell}\leq n-t$. We then have
\begin{eqnarray*}
|\lambda(\alpha_{\ell}, u_{\ell})| &=& 2 \ell + 2 u_{\ell}t = (n-t)^2 + (n-t) +2u_{\ell}t \\
&\ge&  (n-t)^2 + (n-t)+  t(n-t)  \\
&\ge&  (n-t)^2
\end{eqnarray*}
where the first equality comes from $\ell = 2u_{\ell}^2 - u_{\ell} +|\alpha_{\ell}| $ and $\frac{1}{2}|\lambda(\alpha_{\ell}, u_{\ell})| = 2u_{\ell}^2+u_{\ell}(t-1) + |\alpha_{\ell}|$. In particular, there is an $(n-t)\times (n-t)$ square sitting inside $\lambda(\alpha_{\ell}, u_{\ell})$. This proves $\rank \lambda(\alpha_{\ell}, u_{\ell})=n-t$. Similarly, one can show that $\rank {\lambda(\alpha_{\ell-1}, u_{\ell-1})}=n-t$, as well.

$(2)$~
Since $\rank (\alpha_{\ell}, u_{\ell}) = n-t$, we have 
$$
|\lambda(\alpha_{\ell}, u_{\ell})| = (n-t)^2 + (n-t) + 2u_{\ell}t =  (n-t)^2 + (n-t)+  t(n-t) =(n-t)(n+1).
$$
This implies $|\alpha_{\ell}| = n-t$, because $|\lambda(\alpha_{\ell}, u_{\ell})| = (n-t)^2 + (n-t)(t-1) + 2 |\alpha_{\ell}|$. Hence the Young diagram of $\lambda(\alpha_{\ell}, u_{\ell})$ has the following form, for example in the case $n=8$ and $t=4$.
$$
\ytableausetup{mathmode, boxsize=1.2em}
\begin{ytableau}
 {}  &  &  & & \mathbb{X}& \mathbb{X}& \mathbb{X}&\bullet  \\
   &   &  & & \mathbb{X}& \mathbb{X}& \mathbb{X}& \bullet  \\
   &  & & & \mathbb{X}& \mathbb{X}& \mathbb{X}& \bullet \\
   &  & & & \mathbb{X}& \mathbb{X} &\mathbb{X}& \bullet\\
 \circ   &  \circ& \circ& \circ  \\
\end{ytableau}
$$
Therefore
$$
\lambda(\alpha_{\ell}, u_{\ell})=(\underbrace{n, n, \ldots, n}_\text{$n-t$ times}, n-t).
$$
Note that we have the isomorphisms
$$
L_{\lambda(\alpha_{\ell}, u_{\ell})} E \cong L_{(t)} E^* \otimes_{\sk}  \left(\bigwedge^n E\right)^{\oplus (n-t+1)} \cong L_{(t)} E^* \cong \bigwedge^t E^*
$$
where the first isomorphism follows from \cite[Chapter 2, Exercise 18]{W} and $E\cong E^{**}$. 
Hence
$$
\rank L_{\lambda(\alpha_{\ell}, u_{\ell})}E = \rank \ \bigwedge^t E^* = {n\choose t}.
$$ 
On the other hand, since $|\lambda(\alpha_{\ell-1}, u_{\ell-1})|=(n+1)(n-t)-2$, we have $|\alpha_{\ell-1}|=n-t-1$ and thus 
$$
\lambda(\alpha_{\ell-1}, u_{\ell-1})=(\underbrace{n, n, \ldots, n}_\text{$n-t-1$ times}, n-1, n-t-1).
$$
Hence the Young diagram of $\lambda(\alpha_{\ell}, u_{\ell})$ has the following form, for example in the case $n=8$ and $t=4$.
$$
\ytableausetup{mathmode, boxsize=1.2em}
\begin{ytableau}
 {}  &  &  & & \mathbb{X}&  \mathbb{X}&  \mathbb{X}&\bullet  \\
   &   &  & &  \mathbb{X}&  \mathbb{X}&  \mathbb{X}& \bullet  \\
   &  & & &  \mathbb{X}&  \mathbb{X}&  \mathbb{X}& \bullet \\
   &  & & &  \mathbb{X} &  \mathbb{X} & \mathbb{X}\\
 \circ   &  \circ& \circ  \\
\end{ytableau}
$$
By \cite[Chapter 2, Exercise 18]{W} and $E\cong E^{**}$, we have the isomorphisms
$$
L_{\lambda(\alpha_{\ell-1}, u_{\ell-1})} E
\cong L_{(t+1, 1)} E^*\otimes_{\sk} \left(\bigwedge^n E\right)^{\oplus (n-t+1)}\cong L_{(t+1, 1)} E^*. 
$$
Since $L_{(t+1, 1)} E^* = \left(\bigwedge^{t+1} E^* \otimes_{\sk} E^* \right)/R((t+1, 1), E^*)$ and the submodule $R((t+1, 1), E^*)$ is generated by the image of the diagonal injective map $\Delta: \bigwedge^{t+2} E^* \to \bigwedge^{t+1} E^* \otimes_{\sk} E^*$, we conclude that
$$
\rank  L_{\lambda(\alpha_{\ell-1}, u_{\ell-1})}E = \rank L_{(t+1, 1)} E^* = n {n \choose t+1}-{n \choose {t+2}}.
$$
Hence, by Theorem~\ref{rankthm}, we obtain 
$$
\rank F_{\ell} = \binom{n}{t} \ \ \text{and} \ \ \rank F_{\ell-1} = n \binom{n}{t+1}- \binom{n}{t+2}
$$
as desired. 
\end{proof}

We are now ready to prove Theorem \ref{mainthm}. 

\begin{proof}[Proof of Theorem \ref{mainthm}]
$(1)\Rightarrow(2)$~ We choose an exact sequence
$
0 \to R \to K_R(-a)\to C\to 0
$
of graded $R$-modules with $\mu_R(C)={e}_{\mathfrak{m}}^0(C)$, where $a$ denotes the $a$-invariant of $R$. Since $[K_R]_{\mathfrak{m}}\cong K_{R_{\mathfrak{m}}}$ and $C_{\mathfrak{m}}$ is Ulrich as an $R_{\mathfrak{m}}$-module, the local ring $R_{\mathfrak{m}}$ is almost Gorenstein. 

$(3)\Rightarrow(1)$~We may assume $n=3$ and $t=1$. Remember that $R$ is Cohen-Macaulay integral domain with $\dim R=3$ (\cite[Theorem 1]{Kutz}). By taking a parameter ideal $Q=(x_{11},x_{12}, x_{33})$ of $R$, we have $Q\neq \mathfrak{m}$ and $\mathfrak{m}^2=Q\mathfrak{m}$. Hence, $a=1-d=-2$. Then, by Theorem \ref{thm2}, $R$ is an almost Gorenstein graded ring. 

$(2)\Rightarrow(3)$~ 
Set $A=R_{\mathfrak{m}}$ and $\mathfrak{n}=\mathfrak{m}A$. We choose an exact sequence 
$$
0\rightarrow A \xrightarrow{\varphi} K_A \rightarrow C\rightarrow 0
$$
of $A$-modules such that $\mu_A(C) = e_{\mathfrak{n}}^0(C)$. We may assume $n-t$ is even, i.e., $A$ is not a Gorenstein ring. Then, because $\varphi(1)\not\in\mathfrak{n}K_A$, we get $ \mu_A(C)= r(A)-1$ and
$$
0 \to \mathfrak{n} \varphi(1) \to \mathfrak{n} K_A \to \mathfrak{n} C \to 0
$$
is an exact sequence of $A$-modules. Here $r(A)$ denotes the Cohen-Macaulay type of $A$. Hence we get an inequalities
$$
\mu_A(\mathfrak{n}K_A) \leq \mu_A(\mathfrak{n})+\mu_A(\mathfrak{n}C) \leq \dfrac{n(n+1)}{2}+(d-1)(r(A)-1)
$$
where the second inequality follows from $\mathfrak{n} C = (f_1, f_2, \ldots, f_{d-1})C$ for some $f_i \in \mathfrak{n}$. This choice is possible, because $\sk$ is an infinite field; see \cite[Proposition 2.2 (2)]{GTT}.

Let 
$$
0 \to F_{\ell} \to F_{{\ell}-1} \to \cdots \to F_0 \to R \to 0
$$
be a graded minimal $S$-free resolution of $R$. By taking the $S$-dual of the resolution, we get the presentation of the graded canonical module $K_R$. If ${\rm char}(\sk)=0$, then
$$
(\dagger) \ \ \ \mu_R(\mathfrak{m}K_R)\geq \mu_R(\mathfrak{m})\cdot \rank F_{\ell}- \rank F_{{\ell}-1}
$$
because all the entries of the matrix corresponding to the map $F_{\ell} \to F_{{\ell}-1}$ have degree one; see \cite[Proposition 2.5]{T}. Note that the Hilbert series of $R$ does not depend on the field, so is the Hilbert series of $K_R$. Since $R$ is homogeneous and level, we conclude that $\mu_R(\mathfrak{m}K_R)$ does not depend on the characteristic of the field $k$. Hence the above inequality $(\dagger)$ holds for any characteristic of $k$.
Therefore we get the inequality 
$$
\dfrac{n(n+1)}{2}\dbinom{n}{t}-n\dbinom{n}{t+1}+\dbinom{n}{t+2}\leq \dfrac{n(n+1)}{2}+(d-1)\left(\dbinom{n}{t}-1\right). 
$$   
Putting $d=nt-\dfrac{t(t-1)}{2}$, we have 
$$
(\sharp) \ \ \  \left(\binom{n}{t}-1\right)\left( \dfrac{n(n+1)}{2}-nt+\dfrac{t(t-1)}{2}+1\right)\leq n\dbinom{n}{t+1}-\dbinom{n}{t+2}. 
$$
By multiplying both sides by $(t+2)!$ in the above inequality, we obtain
$$
t(t+1)\leq 6
$$ 
which yields $t\leq 2$. Suppose $t=2$. Then
$$
\left(\binom{n}{2}-1\right)\left(\dfrac{n(n+1)}{2}-2n+2\right)\leq n\dbinom{n}{3}-\dbinom{n}{4}.
$$
Hence $n^2-5n+8\leq 0$ by multiplying $4!$ in the above inequality. This makes a contradiction. Therefore $t=1$. By $(\sharp)$, a direct computation shows $n-t \le 2$. As $n-t$ is even, then $n-t=2$, and hence $n=3$. 
\end{proof}

As a consequence of Theorem~\ref{mainthm}, we have the following.

\begin{corollary}
The local ring $\sk[[X]]/{I_{t+1}(X)}$ is almost Gorenstein  if and only if either $n-t$ is odd or $n=3$, $t=1$.
\end{corollary}

\end{document}